\documentclass[11pt]{article}
\usepackage{amsmath,amsthm,amsopn,amstext,amscd,amsfonts,amssymb,verbatim}
\usepackage{pdflscape}
\usepackage{graphicx}
\usepackage[numbers]{natbib}
\usepackage{color}

\def \u {\mathop{\rm \mathcal{U}}\nolimits}
\def \tr {\mathop{\rm tr}\nolimits}
\def \re {\mathop{\rm Re}\nolimits}

\def \eig {\mathop{\rm eig}\nolimits}

\def \Vol {\mathop{\rm Vol}\nolimits}

\def \etr {\mathop{\rm etr}\nolimits}
\def \diag {\mathop{\rm diag}\nolimits}

\renewenvironment{abstract}
                 {\vspace{6pt}
                  \begin{center}
                  \begin{minipage}{5in}
                  \textbf{Abstract}
                  \noindent\ignorespaces
                 }
                 {\end{minipage}\end{center}}
\newtheorem{theorem}{\textbf{Theorem}}[section]
\newtheorem{corollary}{\textbf{Corollary}}[section]

\newtheorem{proposition}{\textbf{Proposition}}[section]
\theoremstyle{definition}
\newtheorem{definition}{\textbf{Definition}}[section]

\newtheorem{remark}{\textbf{Remark}}[section]

\title{\Large \textbf{Generalised matrix multivariate Pearson type II-distribution}}
\author{
  Jos\'e A. D\'{\i}az-Garc\'{\i}a \thanks{Corresponding author: Jos\'e A. D\'{\i}az-Garc\'{\i}a,
  Universidad Aut\'onoma Agraria Antonio Narro, Calzada Antonio Narro 1923, Col. Buenavista,
  25315 Saltillo, Coahuila, M\'exico, E-mail: jadiaz@uaaan.mx}\\
  \emph{Universidad Aut\'onoma Agraria Antonio Narro} \\[2ex]
  Ram\'on Guti\'errez-S\'anchez \thanks{Department of Statistics and O.R, University of Granada, Granada
  18071, Spain, E-mail:ramongs@ugr.es}\\
  \emph{University of Granada}
}
\date{}
\begin{document}
\maketitle

\begin{abstract}
Matrix multivariate Pearson type II-Riesz distribution is defined and some of its properties are studied.
In particular, the associated matrix multivariate beta distribution type I is derived. Also the singular
values and eigenvalues distributions are obtained.
\end{abstract}

\begin{center}
\begin{minipage}{5in}
\end{minipage}
\end{center}

\section{Introduction}\label{sec1}
When a new statistic theory is proposed, the statistician known well about the rigourously mathematical
foundations of their discipline, however in order to reach a wider interdisciplinary public, some of the
classical statistical techniques have been usually published without explaining the supporting  abstract
mathematical tools which governs the approach. For example, in the context of the distribution theory of
random matrices, in the last 20 years, a number of more abstract and mathematical approaches have emerged
for studying and generalising the usual matrix variate distributions. In particular, this needing have
appeared recently in the generalisation, by using abstract algebra, of some results of real random
matrices to another supporting fields, such as complex, quaternion and octonion, see \citet{rva:05a},
\citet{rva:05b}, \citet{er:05}, \citet{f:05}, among many others authors. Studying distribution theory by
another algebras, beyond real, have led several generalisations of substantial understanding in the
theoretical context, and we expect that it is more extensively applied when a an improvement of its
unified potential can be explored in other contexts. Two main tendencies have been considered in
literature, Jordan algebras and real normed division algebras. Some works dealing the first approach are
due to \citet{fk:94}, \citet{m:94}, \citet{cl:96}, \citet{hl:01}, \citet{hlz:05, hlz:08}, \citet{k:14},
and the references therein, meanwhile, the second technique has been studied by \citet{gr:87},
\citet{dggj:11}, \citet{dg:15a,dg:15b,dg:15c}, among many others.

In the same manner, different generalisations of the multivariate statistical analysis have been proposed
recently. This generalised technique studies the effect of changing the usual matrix multivariate normal
support by a general matrix multivariate family of distributions, such as the elliptical contoured
distributions (or simply, matrix multivariate elliptical distributions), see \citet{fz:90} and
\citet{gv:93}. This family of distributions involves a number of known matrix multivariate distributions
such as normal, Kotz type, Bessel, Pearson type II and VII, contaminated normal and power exponential,
among many others. Two important properties of these distributions must be emphasise: i). Matrix
multivariate elliptical distributions provide more flexibility in the statistical modeling by including
distributions with heavier or lighter tails and/or greater or lower degree of kurtosis than matrix
multivariate normal distribution; and, ii). Most of the statistical tests based on matrix multivariate
normal distribution are invariant under the complete family of matrix multivariate elliptical
distributions.

Recently, a slight combination of these two theoretical generalisations have appeared in literature;
namely, Jordan algebras has been led to the matrix multivariate Riesz distribution and its associated
beta distribution.  \citet{dg:15c} proved that the above mentioned distributions can be derived from a
particular matrix multivariate elliptical distribution, termed matrix multivariate Kotz-Riesz
distribution. Similarly, matrix multivariate Riesz distribution is also of interest from the mathematical
point of view; in fact most of their basic properties under \emph{the structure theory of normal
$j$-algebras}  and  \emph{ the theory of Vinberg algebras} in place of Jordan algebras have been studied
by \citet{i:00} and \citet{bh:09}, respectively.

In this scenario, we can now propose a generalisation of the matrix multivariate beta, T and Pearson type
II distributions based on a matrix multivariate Kotz-Riesz distribution. As usual in the normal case,
extensions of beta, T and Pearson type II distributions involves two alternatives, the matricvariate and
the matrix multivariate versions\footnote{The term matricvariate distribution was first introduced
\citet{di:67}, but the expression matrix-variate distribution or matrix variate distribution or matrix
multivariate distribution was later used to describe any distribution of a random matrix, see
\citet{gn:00} and \citet{gv:93}, and the references therein. When the density function of a random matrix
is written including the trace operator then the matrix multivariate designation shall be used. }, see
\citet{dg:15a,dg:15b,dg:15c}, \citet{dggj:11,dggj:12,dggj:13} and \citet{dggs:15}.

This article derives the matrix multivariate beta and Pearson type II distributions obtained from a
matrix multivariate Kots-Riesz distribution and some of their basic properties are studied. Section
\ref{sec2} gives some basic concepts and the notation of abstract algebra, Jacobians and distribution
theory. The nonsingular central matrix multivariate Pearson type II-Riesz distribution and the
corresponding generalised matrix multivariate beta type I distribution are studied in Section \ref{sec3}.
Finally, the joint densities of the singular values are derived in Section \ref{sec4}.

\section{Preliminary results}\label{sec2}

A detailed discussion of real normed division algebras can be found in \citet{b:02} and
\citet{E:90}. For your convenience, we shall introduce some notation, although in general, we
adhere to standard notation forms.

For our purposes: Let $\mathbb{F}$ be a field. An \emph{algebra} $\mathfrak{A}$ over $\mathbb{F}$ is a
pair $(\mathfrak{A};m)$, where $\mathfrak{A}$ is a \emph{finite-dimensional vector space} over
$\mathbb{F}$ and \emph{multiplication} $m : \mathfrak{A} \times \mathfrak{A} \rightarrow A$ is an
$\mathbb{F}$-bilinear map; that is, for all $\lambda \in \mathbb{F},$ $x, y, z \in \mathfrak{A}$,
\begin{eqnarray*}
% \nonumber to remove numbering (before each equation)
  m(x, \lambda y + z) &=& \lambda m(x; y) + m(x; z) \\
  m(\lambda x + y; z) &=& \lambda m(x; z) + m(y; z).
\end{eqnarray*}
Two algebras $(\mathfrak{A};m)$ and $(\mathfrak{E}; n)$ over $\mathbb{F}$ are said to be
\emph{isomorphic} if there is an invertible map $\phi: \mathfrak{A} \rightarrow \mathfrak{E}$ such that
for all $x, y \in \mathfrak{A}$,
$$
  \phi(m(x, y)) = n(\phi(x), \phi(y)).
$$
By simplicity, we write $m(x; y) = xy$ for all $x, y \in \mathfrak{A}$.

Let $\mathfrak{A}$ be an algebra over $\mathbb{F}$. Then $\mathfrak{A}$ is said to be
\begin{enumerate}
  \item \emph{alternative} if $x(xy) = (xx)y$ and $x(yy) = (xy)y$ for all $x, y \in \mathfrak{A}$,
  \item \emph{associative} if $x(yz) = (xy)z$ for all $x, y, z \in \mathfrak{A}$,
  \item \emph{commutative} if $xy = yx$ for all $x, y \in \mathfrak{A}$, and
  \item \emph{unital} if there is a $1 \in \mathfrak{A}$ such that $x1 = x = 1x$ for all $x \in \mathfrak{A}$.
\end{enumerate}
If $\mathfrak{A}$ is unital, then the identity 1 is uniquely determined.

An algebra $\mathfrak{A}$ over $\mathbb{F}$ is said to be a \emph{division algebra} if $\mathfrak{A}$ is
nonzero and $xy = 0_{\mathfrak{A}} \Rightarrow x = 0_{\mathfrak{A}}$ or $y = 0_{\mathfrak{A}}$ for all
$x, y \in \mathfrak{A}$.

The term ``division algebra", comes from the following proposition, which shows that, in such an algebra,
left and right division can be unambiguously performed.

Let $\mathfrak{A}$ be an algebra over $\mathbb{F}$. Then $\mathfrak{A}$ is a division algebra if, and
only if, $\mathfrak{A}$ is nonzero and for all $a, b \in \mathfrak{A}$, with $b \neq 0_{\mathfrak{A}}$,
the equations $bx = a$ and $yb = a$ have unique solutions $x, y \in \mathfrak{A}$.

In the sequel we assume $\mathbb{F} = \Re$ and consider classes of division algebras over $\Re$ or
``\emph{real division algebras}" for short.

We introduce the algebras of \emph{real numbers} $\Re$, \emph{complex numbers} $\mathfrak{C}$,
\emph{quaternions} $\mathfrak{H}$ and \emph{octonions} $\mathfrak{O}$. Then, if $\mathfrak{A}$ is an
alternative real division algebra, then $\mathfrak{A}$ is isomorphic to $\Re$, $\mathfrak{C}$,
$\mathfrak{H}$ or $\mathfrak{O}$.

Let $\mathfrak{A}$ be a real division algebra with identity $1$. Then $\mathfrak{A}$ is said to be
\emph{normed} if there is an inner product $(\cdot, \cdot)$ on $\mathfrak{A}$ such that
$$
  (xy, xy) = (x, x)(y, y) \qquad \mbox{for all } x, y \in \mathfrak{A}.
$$
If $\mathfrak{A}$ is a \emph{real normed division algebra}, then $\mathfrak{A}$ is isomorphic $\Re$,
$\mathfrak{C}$, $\mathfrak{H}$ or $\mathfrak{O}$.

There are exactly four normed division algebras: real numbers ($\Re$), complex numbers ($\mathfrak{C}$),
quaternions ($\mathfrak{H}$) and octonions ($\mathfrak{O}$), see \citet{b:02}. We take into account that
should be taken into account, $\Re$, $\mathfrak{C}$, $\mathfrak{H}$ and $\mathfrak{O}$ are the only
normed division algebras; furthermore, they are the only alternative division algebras.

Let $\mathfrak{A}$ be a division algebra over the real numbers. Then $\mathfrak{A}$ has dimension either
1, 2, 4 or 8. Finally, observe that

\begin{tabular}{c}
  $\Re$ is a real commutative associative normed division algebras, \\
  $\mathfrak{C}$ is a commutative associative normed division algebras,\\
  $\mathfrak{H}$ is an associative normed division algebras, \\
  $\mathfrak{O}$ is an alternative normed division algebras. \\
\end{tabular}

Let $\mathfrak{L}^{\beta}_{n,m}$ be the set of all $n \times m$ matrices of rank $m \leq n$ over
$\mathfrak{A}$ with $m$ distinct positive singular values, where $\mathfrak{A}$ denotes a \emph{real
finite-dimensional normed division algebra}. Let $\mathfrak{A}^{n \times m}$ be the set of all $n \times
m$ matrices over $\mathfrak{A}$. The dimension of $\mathfrak{A}^{n \times m}$ over $\Re$ is $\beta mn$.
Let $\mathbf{A} \in \mathfrak{A}^{n \times m}$, then $\mathbf{A}^{*} = \bar{\mathbf{A}}^{T}$ denotes the
usual conjugate transpose.

Table \ref{table1} sets out the equivalence between the same concepts in the four normed division
algebras.

\begin{table}[th]
  \centering
  \caption{\scriptsize Notation}\label{table1}
  \begin{scriptsize}
  \begin{tabular}{cccc|c}
    \hline
    % after \\: \hline or \cline{col1-col2} \cline{col3-col4} ...
    Real & Complex & Quaternion & Octonion & \begin{tabular}{c}
                                               Generic \\
                                               notation \\
                                             \end{tabular}\\
    \hline
    Semi-orthogonal & Semi-unitary & Semi-symplectic & \begin{tabular}{c}
                                                         Semi-exceptional \\
                                                         type \\
                                                       \end{tabular}
      & $\mathcal{V}_{m,n}^{\beta}$ \\
    Orthogonal & Unitary & Symplectic & \begin{tabular}{c}
                                                         Exceptional \\
                                                         type \\
                                                       \end{tabular} & $\mathfrak{U}^{\beta}(m)$ \\
    Symmetric & Hermitian & \begin{tabular}{c}
                              % after \\: \hline or \cline{col1-col2} \cline{col3-col4} ...
                              Quaternion \\
                              hermitian \\
                            \end{tabular}
     & \begin{tabular}{c}
                              % after \\: \hline or \cline{col1-col2} \cline{col3-col4} ...
                              Octonion \\
                              hermitian \\
                            \end{tabular} & $\mathfrak{S}_{m}^{\beta}$ \\
    \hline
  \end{tabular}
  \end{scriptsize}
\end{table}

We denote by ${\mathfrak S}_{m}^{\beta}$ the real vector space of all $\mathbf{S} \in \mathfrak{A}^{m
\times m}$ such that $\mathbf{S} = \mathbf{S}^{*}$. In addition, let $\mathfrak{P}_{m}^{\beta}$ be the
\emph{cone of positive definite matrices} $\mathbf{S} \in \mathfrak{A}^{m \times m}$. Thus,
$\mathfrak{P}_{m}^{\beta}$ consist of all matrices $\mathbf{S} = \mathbf{X}^{*}\mathbf{X}$, with
$\mathbf{X} \in \mathfrak{L}^{\beta}_{n,m}$; then $\mathfrak{P}_{m}^{\beta}$ is an open subset of
${\mathfrak S}_{m}^{\beta}$.

Let $\mathfrak{D}_{m}^{\beta}$ consisting of all $\mathbf{D} \in \mathfrak{A}^{m \times m}$, $\mathbf{D}
= \diag(d_{1}, \dots,d_{m})$. Let $\mathfrak{T}_{U}^{\beta}(m)$ be the subgroup of all \emph{upper
triangular} matrices $\mathbf{T} \in \mathfrak{A}^{m \times m}$ such that $t_{ij} = 0$ for $1 < i < j
\leq m$. Let $\mathbf{Z} \in \mathfrak{L}^{\beta}_{n,m}$, define the norm of $\mathbf{Z}$ as
$||\mathbf{Z}|| = \sqrt{\tr \mathbf{Z}^{*}\mathbf{Z}}$.

For any matrix $\mathbf{X} \in \mathfrak{A}^{n \times m}$, $d\mathbf{X}$ denotes the\emph{ matrix of
differentials} $(dx_{ij})$. Finally, we define the \emph{measure} or volume element $(d\mathbf{X})$ when
$\mathbf{X} \in \mathfrak{A}^{n \times m}, \mathfrak{S}_{m}^{\beta}$, $\mathfrak{D}_{m}^{\beta}$ or
$\mathcal{V}_{m,n}^{\beta}$, see \citet{dggj:11} and \citet{dggj:13}.

If $\mathbf{X} \in \mathfrak{A}^{n \times m}$ then $(d\mathbf{X})$ (the Lebesgue measure in
$\mathfrak{A}^{n \times m}$) denotes the exterior product of the $\beta mn$ functionally independent
variables
$$
  (d\mathbf{X}) = \bigwedge_{i = 1}^{n}\bigwedge_{j = 1}^{m}dx_{ij} \quad \mbox{ where }
    \quad dx_{ij} = \bigwedge_{k = 1}^{\beta}dx_{ij}^{(k)}.
$$

If $\mathbf{S} \in \mathfrak{S}_{m}^{\beta}$ (or $\mathbf{S} \in \mathfrak{T}_{U}^{\beta}(m)$ with
$t_{ii} >0$, $i = 1, \dots,m$) then $(d\mathbf{S})$ (the Lebesgue measure in $\mathfrak{S}_{m}^{\beta}$
or in $\mathfrak{T}_{U}^{\beta}(m)$) denotes the exterior product of the exterior product of the
$m(m-1)\beta/2 + m$ functionally independent variables,
$$
  (d\mathbf{S}) = \bigwedge_{i=1}^{m} ds_{ii}\bigwedge_{i > j}^{m}\bigwedge_{k = 1}^{\beta}
                      ds_{ij}^{(k)}.
$$
Observe, that for the Lebesgue measure $(d\mathbf{S})$ defined thus, it is required that $\mathbf{S} \in
\mathfrak{P}_{m}^{\beta}$, that is, $\mathbf{S}$ must be a non singular Hermitian matrix (Hermitian
definite positive matrix).

If $\mathbf{\Lambda} \in \mathfrak{D}_{m}^{\beta}$ then $(d\mathbf{\Lambda})$ (the Legesgue measure in
$\mathfrak{D}_{m}^{\beta}$) denotes the exterior product of the $\beta m$ functionally independent
variables
$$
  (d\mathbf{\Lambda}) = \bigwedge_{i = 1}^{n}\bigwedge_{k = 1}^{\beta}d\lambda_{i}^{(k)}.
$$
If $\mathbf{H}_{1} \in \mathcal{V}_{m,n}^{\beta}$ then
$$
  (\mathbf{H}^{*}_{1}d\mathbf{H}_{1}) = \bigwedge_{i=1}^{m} \bigwedge_{j =i+1}^{n}
  \mathbf{h}_{j}^{*}d\mathbf{h}_{i}.
$$
where $\mathbf{H} = (\mathbf{H}^{*}_{1}|\mathbf{H}^{*}_{2})^{*} = (\mathbf{h}_{1}, \dots,
\mathbf{h}_{m}|\mathbf{h}_{m+1}, \dots, \mathbf{h}_{n})^{*} \in \mathfrak{U}^{\beta}(n)$. It can be
proved that this differential form does not depend on the choice of the $\mathbf{H}_{2}$ matrix. When $n
= 1$; $\mathcal{V}^{\beta}_{m,1}$ defines the unit sphere in $\mathfrak{A}^{m}$. This is, of course, an
$(m-1)\beta$- dimensional surface in $\mathfrak{A}^{m}$. When $n = m$ and denoting $\mathbf{H}_{1}$ by
$\mathbf{H}$, $(\mathbf{H}d\mathbf{H}^{*})$ is termed the \emph{Haar measure} on
$\mathfrak{U}^{\beta}(m)$.

The surface area or volume of the Stiefel manifold $\mathcal{V}^{\beta}_{m,n}$ is
\begin{equation}\label{vol}
    \Vol(\mathcal{V}^{\beta}_{m,n}) = \int_{\mathbf{H}_{1} \in
  \mathcal{V}^{\beta}_{m,n}} (\mathbf{H}_{1}d\mathbf{H}^{*}_{1}) =
  \frac{2^{m}\pi^{mn\beta/2}}{\Gamma^{\beta}_{m}[n\beta/2]},
\end{equation}
where $\Gamma^{\beta}_{m}[a]$ denotes the multivariate \emph{Gamma function} for the space
$\mathfrak{S}_{m}^{\beta}$ and is defined as
\begin{eqnarray*}
% \nonumber to remove numbering (before each equation)
  \Gamma_{m}^{\beta}[a] &=& \displaystyle\int_{\mathbf{A} \in \mathfrak{P}_{m}^{\beta}}
  \etr\{-\mathbf{A}\} |\mathbf{A}|^{a-(m-1)\beta/2 - 1}(d\mathbf{A}) \\ \label{cgamma}
    &=& \pi^{m(m-1)\beta/4}\displaystyle\prod_{i=1}^{m} \Gamma[a-(i-1)\beta/2],
\end{eqnarray*}
and $\re(a)> (m-1)\beta/2$. This can be obtained as a particular case of the \emph{generalised gamma
function of weight $\kappa$} for the space $\mathfrak{S}^{\beta}_{m}$ with $\kappa = (k_{1}, k_{2},
\dots, k_{m}) \in \Re^{m}$, taking $\kappa =(0,0,\dots,0) \in \Re^{m}$ and which for $\re(a) \geq
(m-1)\beta/2 - k_{m}$ is defined by, see \citet{gr:87} and \citet{fk:94},
\begin{eqnarray}\label{int1}
  \Gamma_{m}^{\beta}[a,\kappa] &=& \displaystyle\int_{\mathbf{A} \in \mathfrak{P}_{m}^{\beta}}
  \etr\{-\mathbf{A}\} |\mathbf{A}|^{a-(m-1)\beta/2 - 1} q_{\kappa}(\mathbf{A}) (d\mathbf{A}) \\
&=& \pi^{m(m-1)\beta/4}\displaystyle\prod_{i=1}^{m} \Gamma[a + k_{i}
    -(i-1)\beta/2]\nonumber\\ \label{gammagen1}
&=& [a]_{\kappa}^{\beta} \Gamma_{m}^{\beta}[a],
\end{eqnarray}
where $\etr(\cdot) = \exp(\tr(\cdot))$, $|\cdot|$ denotes the determinant, and for $\mathbf{A} \in
\mathfrak{S}_{m}^{\beta}$
\begin{equation}\label{hwv}
    q_{\kappa}(\mathbf{A}) = |\mathbf{A}_{m}|^{k_{m}}\prod_{i = 1}^{m-1}|\mathbf{A}_{i}|^{k_{i}-k_{i+1}}
\end{equation}
with $\mathbf{A}_{p} = (a_{rs})$, $r,s = 1, 2, \dots, p$, $p = 1,2, \dots, m$ is termed the \emph{highest
weight vector}, see \citet{gr:87},  \citet{fk:94} and \citet{hl:01}; And, $[a]_{\kappa}^{\beta}$ denotes
the generalised Pochhammer symbol of weight $\kappa$, defined as
\begin{eqnarray*}
% \nonumber to remove numbering (before each equation)
  [a]_{\kappa}^{\beta} &=& \prod_{i = 1}^{m}(a-(i-1)\beta/2)_{k_{i}}\\
    &=& \frac{\pi^{m(m-1)\beta/4} \displaystyle\prod_{i=1}^{m}
    \Gamma[a + k_{i} -(i-1)\beta/2]}{\Gamma_{m}^{\beta}[a]} \\
    &=& \frac{\Gamma_{m}^{\beta}[a,\kappa]}{\Gamma_{m}^{\beta}[a]},
\end{eqnarray*}
where $\re(a) > (m-1)\beta/2 - k_{m}$ and
$$
  (a)_{i} = a (a+1)\cdots(a+i-1),
$$
is the standard Pochhammer symbol.

Additional, note that, if $\kappa = (p, \dots, p)$, then $ q_{\kappa}(\mathbf{A}) = |\mathbf{A}|^{p}$. In
particular if $p=0$, then $q_{\kappa}(\mathbf{A}) = 1$. If $\tau = (t_{1}, t_{2}, \dots, t_{m})$,
$t_{1}\geq t_{2}\geq \cdots \geq t_{m} \geq 0$, then $q_{\kappa+\tau}(\mathbf{A}) =
q_{\kappa}(\mathbf{A})q_{\tau}(\mathbf{A})$, and in particular if $\tau = (p,p, \dots, p)$,  then $
q_{\kappa+\tau}(\mathbf{A}) \equiv q_{\kappa+p}(\mathbf{A}) = |\mathbf{A}|^{p} q_{\kappa}(\mathbf{A}) $.
Finally, for $\mathbf{B} \in \mathfrak{T}_{U}^{\beta}(m)$  in such a manner that $\mathbf{C} =
\mathbf{B}^{*}\mathbf{B} \in \mathfrak{S}_{m}^{\beta}$, $q_{\kappa}(\mathbf{B}^{*}\mathbf{AB}) =
q_{\kappa}(\mathbf{C})q_{\kappa}(\mathbf{A})$, and $q_{\kappa}(\mathbf{B}^{*-1}\mathbf{A}\mathbf{B}^{-1})
= (q_{\kappa}(\mathbf{C}))^{-1}q_{\kappa}(\mathbf{A}) = q_{-\kappa}(\mathbf{C})q_{\kappa}(\mathbf{A}), $
see \citet{hlz:08}.

Finally, the following Jacobians involving the $\beta$ parameter, reflects the generalised power of the
algebraic technique; the can be seen as extensions of the full derived and unconnected results in the
real, complex or quaternion cases, see \citet{fk:94} and \citet{dggj:11}. These results are the base for
several matrix and matric variate generalised analysis.

\begin{proposition}\label{lemlt}
Let $\mathbf{X}$ and $\mathbf{Y} \in \mathfrak{L}_{n,m}^{\beta}$  be matrices of functionally independent
variables, and let $\mathbf{Y} = \mathbf{AXB} + \mathbf{C}$, where $\mathbf{A} \in
\mathfrak{L}_{n,n}^{\beta}$, $\mathbf{B} \in \mathfrak{L}_{m,m}^{\beta}$ and $\mathbf{C} \in
\mathfrak{L}_{n,m}^{\beta}$ are constant matrices. Then
\begin{equation}\label{lt}
    (d\mathbf{Y}) = |\mathbf{A}^{*}\mathbf{A}|^{m\beta/2} |\mathbf{B}^{*}\mathbf{B}|^{
    mn\beta/2}(d\mathbf{X}).
\end{equation}
\end{proposition}

\begin{proposition}[Singular Value Decomposition, $SVD$]\label{lemsvd}
Let $\mathbf{X} \in {\mathcal L}_{n,m}^{\beta}$  be matrix of functionally independent variables, such
that $\mathbf{X} = \mathbf{W}_{1}\mathbf{D}\mathbf{V}^{*}$ \ with \ $\mathbf{W}_{1} \in {\mathcal
V}_{m,n}^{\beta}$, $\mathbf{V} \in \mathfrak{U}^{\beta}(m)$ \ and \ $\mathbf{D} = \diag(d_{1},
\cdots,d_{m}) \in \mathfrak{D}_{m}^{1}$, $d_{1}> \cdots > d_{m} > 0$. Then
\begin{equation}\label{svd}
    (d\mathbf{X}) = 2^{-m}\pi^{\varrho} \prod_{i = 1}^{m} d_{i}^{\beta(n - m + 1) -1}
    \prod_{i < j}^{m}(d_{i}^{2} - d_{j}^{2})^{\beta} (d\mathbf{D}) (\mathbf{V}^{*}d\mathbf{V})
    (\mathbf{W}_{1}^{*}d\mathbf{W}_{1}),
\end{equation}
where
$$
  \varrho = \left\{
             \begin{array}{rl}
               0, & \beta = 1; \\
               -m, & \beta = 2; \\
               -2m, & \beta = 4; \\
               -4m, & \beta = 8.
             \end{array}
           \right.
$$
\end{proposition}

\begin{proposition}\label{lemW}
Let $\mathbf{X} \in \mathfrak{L}_{n,m}^{\beta}$  be matrix of functionally independent variables, and
write $\mathbf{X}=\mathbf{V}_{1}\mathbf{T}$, where $\mathbf{V}_{1} \in {\mathcal V}_{m,n}^{\beta}$ and
$\mathbf{T}\in \mathfrak{T}_{U}^{\beta}(m)$ with positive diagonal elements. Define $\mathbf{S} =
\mathbf{X}^{*}\mathbf{X} \in \mathfrak{P}_{m}^{\beta}.$ Then
\begin{equation}\label{w}
    (d\mathbf{X}) = 2^{-m} |\mathbf{S}|^{\beta(n - m + 1)/2 - 1}
    (d\mathbf{S})(\mathbf{V}_{1}^{*}d\mathbf{V}_{1}).
\end{equation}
\end{proposition}
Finally, to define the matrix multivariate Pearson type II-Riesz distribution we need to recall the
following two definitions of Kotz-Riesz and Riesz distributions.

From  \citet{dg:15c}.
\begin{definition}\label{defKR}
Let $\boldsymbol{\Sigma} \in \boldsymbol{\Phi}_{m}^{\beta}$, $\boldsymbol{\Theta} \in
\boldsymbol{\Phi}_{n}^{\beta}$, $\boldsymbol{\mu} \in \mathfrak{L}^{\beta}_{n,m}$ and  $\kappa = (k_{1},
k_{2}, \dots, k_{m}) \in \Re^{m}$. And let $\mathbf{Y} \in \mathfrak{L}^{\beta}_{n,m}$ and
$\u(\mathbf{B}) \in \mathfrak{T}_{U}^{\beta}(n)$, such that $\mathbf{B} =
\u(\mathbf{B})^{*}\u(\mathbf{B})$ is the Cholesky decomposition of $\mathbf{B} \in
\mathfrak{S}_{m}^{\beta}$. Then it is said that $\mathbf{Y}$ has a \textit{Kotz-Riesz distribution of
type I} and its density function is
$$
    \frac{\beta^{mn\beta/2+\sum_{i = 1}^{m}k_{i}}\Gamma_{m}^{\beta}[n\beta/2]}{\pi^{mn\beta/2}\Gamma_{m}^{\beta}[n\beta/2,\kappa]
    |\boldsymbol{\Sigma}|^{n\beta/2}|\boldsymbol{\Theta}|^{m\beta/2}}\hspace{4cm}
$$
$$\hspace{1cm}
    \times \etr\left\{- \beta\tr \left [\boldsymbol{\Sigma}^{-1} (\mathbf{Y} - \boldsymbol{\mu})^{*}
    \boldsymbol{\Theta}^{-1}(\mathbf{Y} - \boldsymbol{\mu})\right ]\right\}
$$
\begin{equation}\label{dfEKR1}\hspace{3.1cm}
    \times q_{\kappa}\left [\u(\boldsymbol{\Sigma})^{*-1} (\mathbf{Y} - \boldsymbol{\mu})^{*}
    \boldsymbol{\Theta}^{-1}(\mathbf{Y} - \boldsymbol{\mu})\u(\boldsymbol{\Sigma})^{-1}\right ](d\mathbf{Y})
\end{equation}
with $\re(n\beta/2) > (m-1)\beta/2 - k_{m}$;  denoting this fact as
$$
    \mathbf{Y} \sim \mathcal{KR}^{\beta, I}_{n \times m}
    (\kappa,\boldsymbol{\mu}, \boldsymbol{\Theta}, \boldsymbol{\Sigma}).
$$
\end{definition}

From \citet{hl:01} and \citet{dg:15a}.
\begin{definition}\label{defR}
Let $\mathbf{\Xi} \in \mathbf{\Phi}_{m}^{\beta}$ and  $\kappa = (k_{1}, k_{2}, \dots, k_{m}) \in
\Re^{m}$. Then it is said that $\mathbf{V}$ has a \textit{Riesz distribution of type I} if its density
function is
\begin{equation}\label{dfR1}
    \frac{\beta^{am+\sum_{i = 1}^{m}k_{i}}}{\Gamma_{m}^{\beta}[a,\kappa] |\mathbf{\Xi}|^{a}q_{\kappa}(\mathbf{\Xi})}
    \etr\{-\beta\mathbf{\Xi}^{-1}\mathbf{V}\}|\mathbf{V}|^{a-(m-1)\beta/2 - 1}
    q_{\kappa}(\mathbf{V})(d\mathbf{V})
 \end{equation}
for $\mathbf{V} \in \mathfrak{P}_{m}^{\beta}$ and $\re(a) \geq (m-1)\beta/2 - k_{m}$; denoting this fact
as $\mathbf{V} \sim \mathcal{R}^{\beta, I}_{m}(a,\kappa, \mathbf{\Xi})$.
\end{definition}

\section{Matrix multivariate Pearson type II-Riesz distribution}\label{sec3}

A detailed discussion of Riesz distribution may be found in \citet{hl:01} and \citet{dg:15a}. In addition
the Kotz-Riesz distribution is studied in detail in \citet{dg:15c}. For your convenience, we adhere to
standard notation stated in \citet{dg:15a, dg:15c}.

\begin{theorem}\label{teo4}
Let $\left(S_{1}^{1/2}\right)^{2} = S_{1} \sim \mathcal{R}_{1}^{\beta,I}(\nu\beta/2, k,1)$, $k \in \Re$
and $\re(\nu\beta/2)> -k$; independent of $\mathbf{Y} \sim \mathcal{KR}_{n \times
m}^{\beta,I}(\tau,\mathbf{0}, \mathbf{I}_{n}, \mathbf{I}_{m})$, $\re(n\beta/2)> (m-1)\beta/2-t_{m}$. In
addition, define $\mathbf{R} = S^{-1/2}\mathbf{Y}$ where $S = S_{1}+||\mathbf{Y}||^{2}$. Then $S \sim
\mathcal{R}_{1}^{\beta,I}((\nu+mn)\beta/2+\sum_{i=1}^{m}t_{i}, k,1)$ independent of $\mathbf{R}$.
Furthermore the density of $\mathbf{R}$ is
\begin{equation}\label{p2r}
  \frac{\Gamma^{\beta}_{m}[n\beta/2]\Gamma^{\beta}_{1}\left[(\nu+mn)\beta/2+k+\sum_{i=1}^{m}t_{i}\right] }
  {\pi^{\beta mn/2} \Gamma^{\beta}_{m}[n\beta/2,\tau]\Gamma^{\beta}_{1}[\nu\beta/2+k]}
  \left(1-||\mathbf{R}||^{2} \right)^{\nu\beta/2+k-1}
  q_{\tau}\left(\mathbf{R}^{*}\mathbf{R}\right)(d\mathbf{R}),
\end{equation}
where $\left(1-||\mathbf{R}||^{2} \right) > 0$; which is termed the \emph{standardised matrix
multivariate Pearson type II-Riesz type distribution} and is denoted as $\mathbf{R} \sim
\mathcal{P_{_{II}}R}_{m \times n}^{\beta,I}(\nu,k,\tau,1,\boldsymbol{0}, \mathbf{I}_{n}, \mathbf{I}m)$.
\end{theorem}
\begin{proof}
From definition \ref{defKR} and \ref{defR}, the joint density of $S_{1}$ and $\mathbf{Y}$ is
$$
  \propto s_{1}^{\beta \nu/2+k -1} \etr\left\{-\beta \left(s_{1} + ||\mathbf{Y}||^{2}\right)\right\}  q_{\tau}\left(\mathbf{Y}^{*}
  \mathbf{Y}\right)(ds_{1})(d\mathbf{Y})
$$
where the constant of proportionality is
$$
  c = \frac{\beta^{\nu\beta/2+k}}{\Gamma_{1}^{\beta}[\nu\beta/2+k]}
      \ \cdot \ \frac{\beta^{mn\beta/2+\sum_{i=1}^{m}t_{i}} \ \Gamma_{m}^{\beta}[n\beta/2]}{\pi^{mn\beta/2}
      \Gamma_{m}^{\beta}[n\beta/2,\tau]}.
$$
Making the change of variable $S = S_{1}-||\mathbf{Y}||^{2}$ and $\mathbf{Y} = S_{1}^{1/2}\mathbf{R}$, by
(\ref{lt})
$$
  (ds_{1})\wedge(d\mathbf{Y})= s^{\beta mn/2}(ds)\wedge(d\mathbf{R}).
$$
Now, observing that $S = S_{1}-||\mathbf{Y}||^{2} = S\left(1-||\mathbf{R}||^{2}\right)$, the joint
density of $S$ and $\mathbf{R}$ is
$$
  \propto \left(1-||\mathbf{R}||^{2}\right)^{\beta \nu/2+k -1} s^{\beta \nu/2+k -1}
  \etr\left\{-\beta s\right\}  q_{\tau}\left(s\mathbf{R}^{*}
  \mathbf{R}\right)(ds)(d\mathbf{R}).
$$
Also, note that
$$
  q_{\tau}\left(s\mathbf{R}^{*}\mathbf{R}\right) = q_{\tau}\left((s^{1/2}\mathbf{I}_{m})
  \mathbf{R}^{*}\mathbf{R}(s^{1/2}\mathbf{I}_{m})\right) = q_{\tau}\left(s\mathbf{I}_{m}\right)
  q_{\tau}\left(\mathbf{R}^{*}\mathbf{R}\right) = s^{\sum_{i=1}^{m}t_{i}}q_{\tau}\left(\mathbf{R}^{*}
  \mathbf{R}\right).
$$
From where, the joint density of $S$ and $\mathbf{R}$ is
$$
  = \frac{\beta^{(\nu+mn)\beta/2+k+\sum_{i=1}^{m}t_{i}}}{\Gamma^{\beta}_{1}
  \left[(\nu+mn)\beta/2+k+\sum_{i=1}^{m}t_{i}\right]} \etr\left\{-\beta s\right\}
  s^{(\nu+mn)\beta/2+k+\sum_{i=1}^{m}t_{i} -1} (ds)
$$
$$
  \times \ \frac{\Gamma^{\beta}_{m}[n\beta/2]\Gamma^{\beta}_{1}\left[(\nu+mn)\beta/2+k+\sum_{i=1}^{m}t_{i}\right] }
  {\pi^{\beta mn/2} \Gamma^{\beta}_{m}[n\beta/2,\tau]\Gamma^{\beta}_{1}[\nu\beta/2+k]}
  \left(1-||\mathbf{R}||^{2} \right)^{\nu\beta/2+k-1}
  q_{\tau}\left(\mathbf{R}^{*}\mathbf{R}\right)(d\mathbf{R}),
$$
which shows that $S \sim \mathcal{R}_{1}^{\beta,I}((\nu+mn)\beta/2+\sum_{i=1}^{m}t_{i}, k,1)$, and is
independent of $\mathbf{R}$, where $\mathbf{R}$ has the density (\ref{p2r}).
\end{proof}

\begin{corollary}\label{cor1}
Let $\mathbf{R} \sim \mathcal{P_{_{II}}R}_{m \times n}^{\beta,I}(\nu,k,\tau,1,\boldsymbol{0},
\mathbf{I}_{n}, \mathbf{I}m)$ and define
$$
  \mathbf{C}= \rho^{-1/2}\u(\mathbf{\Theta})^{*}\mathbf{R}\u(\mathbf{\Sigma})+\boldsymbol{\mu}
$$
where $\u(\mathbf{B}) \in \mathfrak{T}_{U}^{\beta}(n)$, such that $\mathbf{B} =
\u(\mathbf{B})^{*}\u(\mathbf{B})$ is the Cholesky decomposition of $\mathbf{B} \in
\mathfrak{S}_{m}^{\beta}$, $\mathbf{\Theta} \in \mathfrak{P}_{n}^{\beta}$, $\mathbf{\Sigma} \in
\mathfrak{P}_{m}^{\beta}$, $\rho > 0$ constant and $\boldsymbol{\mu} \in \mathfrak{L}_{n,m}^{\beta}$ is a
matrix of constants. Then the density of $\mathbf{S}$ is
$$
    \propto  \left(1-\rho\tr \mathbf{\Sigma}^{-1}(\mathbf{C}- \boldsymbol{\mu})^{*}
    \mathbf{\Theta}^{-1}(\mathbf{C}- \boldsymbol{\mu}) \right)^{\nu \beta/2+k-1}
$$
\begin{equation}\label{mmtr2}
    \hspace{2cm} \times \ q_{\tau}\left[\u(\mathbf{\Sigma})^{*-1}(\mathbf{C}- \boldsymbol{\mu})^{*}
    \mathbf{\Theta}^{-1}(\mathbf{C}- \boldsymbol{\mu})\u(\mathbf{\Sigma})^{-1}\right] (d\mathbf{S})
\end{equation}
where $\left(1-\rho\tr \mathbf{\Sigma}^{-1}(\mathbf{C}- \boldsymbol{\mu})^{*}
    \mathbf{\Theta}^{-1}(\mathbf{C}- \boldsymbol{\mu}) \right) > 0$; with constant of proportionality
$$
  \frac{\Gamma^{\beta}_{m}[n\beta/2]\Gamma^{\beta}_{1}\left[(\nu+mn)\beta/2-k-\sum_{i=1}^{m}t_{i}\right]
  \rho^{mn\beta /2-\sum_{i=1}^{m}t_{i}}}{\pi^{\beta mn/2}\Gamma^{\beta}_{m}[n\beta/2,-\tau]
  \Gamma^{\beta}_{1}[\nu\beta/2-k]|\mathbf{\Sigma}|^{\beta n/2}|\mathbf{\Theta}|^{\beta m/2}},
$$
which is termed the \emph{matrix multivariate Pearson type II-Riesz distribution} and is denoted as
$\mathbf{C} \sim \mathcal{P_{_{II}}R}_{m \times n}^{\beta,I}(\nu,k,\tau,\rho,\boldsymbol{\mu},
\mathbf{\Theta}, \mathbf{\Sigma})$.
\end{corollary}
\begin{proof}
Observe that $\mathbf{R} = \rho^{1/2} \u(\mathbf{\Theta})^{*-1}(\mathbf{C} -
\boldsymbol{\mu})\u(\mathbf{\Sigma})^{-1}$ and
$$
  (d\mathbf{R}) = \rho^{mn\beta/2} |\mathbf{\Sigma}|^{-\beta n/2}|\mathbf{\Theta}|^{-\beta
  m/2}(d\mathbf{C}).
$$
The desired result is obtained making this change of variable in (\ref{p2r}).
\end{proof}

Next we derive the corresponding \textit{matrix multivariate beta type I distribution}.

\begin{theorem}\label{teo2}
Let $\mathbf{R} \sim \mathcal{P_{_{II}}R}_{n \times m}^{\beta,I}(\nu,
k,\tau,\rho,\boldsymbol{0},\mathbf{I}_{n}, \mathbf{\Sigma})$, and define $\mathbf{B} =
\mathbf{R}^{*}\mathbf{R} \in \mathfrak{P}_{m}^{\beta}$, with $n \geq m$. Then the density of $\mathbf{B}$
is,
  \begin{equation}\label{MMTR1}
    \propto |\mathbf{B}|^{(n-m+1)\beta/2-1} (1-\rho\tr\mathbf{\Sigma}^{-1}\mathbf{B})^{\nu\beta/2+k-1}
    q_{\tau}(\mathbf{B})(d\mathbf{B}),
  \end{equation}
where $1-\rho\tr\mathbf{\Sigma}^{-1}\mathbf{B} > 0$; and with constant of proportionality
  $$
    \frac{\Gamma^{\beta}_{1}\left[(\nu+mn)\beta/2+k+\sum_{i=1}^{m}t_{i}\right] \rho^{\beta mn/2 + \sum_{i=1}^{m}t_{i}}}
    { \Gamma^{\beta}_{m}[n\beta/2,\tau]
    \Gamma^{\beta}_{1}[\nu\beta/2+k]|\mathbf{\Sigma}|^{n\beta/1}q_{\tau}(\mathbf{\Sigma})}.
  $$
$\mathbf{B}$ is said to have a \emph{nonstandarised matrix multivariate beta-Riesz type I distribution}.
\end{theorem}
\begin{proof}
The desired result follows from (\ref{p2r}), by applying (\ref{w}) and then (\ref{vol}); and observing
that $q_{\tau}(\u(\mathbf{\Sigma})^{*-1}\mathbf{B}\u(\mathbf{\Sigma})^{-1}) =
q_{-\tau}(\mathbf{\Sigma})q_{\tau}(\mathbf{B})$.
\end{proof}

In particular if $\mathbf{\Sigma} = \mathbf{I}_{m}$ in Theorem \ref{teo2}, we obtain:

\begin{corollary}\label{cor2}
Let $\mathbf{R} \sim \mathcal{P_{_{II}}R}_{n \times m}^{\beta,I}(\nu,
k,\tau,1,\boldsymbol{0},\mathbf{I}_{n}, \mathbf{I}_{m})$, and define $\mathbf{B} =
\mathbf{R}^{*}\mathbf{R} \in \mathfrak{P}_{m}^{\beta}$, with $n \geq m$. Then the density of $\mathbf{B}$
is,
  \begin{equation}\label{sbr1}
    \frac{\Gamma^{\beta}_{1}\left[(\nu+mn)\beta/2+k+\sum_{i=1}^{m}t_{i}\right]}
    { \Gamma^{\beta}_{m}[n\beta/2,\tau] \Gamma^{\beta}_{1}[\nu\beta/2+k]} |\mathbf{B}|^{(n-m+1)\beta/2-1} (1-\rho\tr\mathbf{B})^{\nu\beta/2+k-1}
    q_{\tau}(\mathbf{B})(d\mathbf{B}),
  \end{equation}
where $1-\rho\tr\mathbf{B} > 0$. $\mathbf{B}$ is said to have a \emph{matrix multivariate beta-Riesz type
I distribution}.
\end{corollary}

\begin{remark}
Observe that alternatively to classical definitions of generalised \textit{matricvariate} beta function
(for symmetric cones), see \citet{dg:15b}, \citet{fk:94} and \citet{hlz:05}, defined as
$$
  \mathcal{B}_{m}^{\beta}[a,\kappa;b, \tau] \hspace{10cm}
$$
$$
  = \int_{\mathbf{0}<\mathbf{S}<\mathbf{I}_{m}}
    |\mathbf{B}|^{b-(m-1)\beta/2-1} q_{\tau}(\mathbf{B})|\mathbf{I}_{m} - \mathbf{B}|^{a-(m-1)\beta/2-1}
    q_{\kappa}(\mathbf{I}_{m} - \mathbf{B})(d\mathbf{B})
$$
$$
   = \int_{\mathbf{F} \in\mathfrak{P}_{m}^{\beta}} |\mathbf{F}|^{b-(m-1)\beta/2-1}
    q_{\tau}(\mathbf{F})|\mathbf{I}_{m} + \mathbf{F}|^{-(a+b)} q_{-(\kappa+\tau)}(\mathbf{I}_{m} + \mathbf{F})
    (d\mathbf{F})
$$
$$
   = \frac{\Gamma_{m}^{\beta}[a,\kappa] \Gamma_{m}^{\beta}[b,\tau]}{\Gamma_{m}^{\beta}[a+b,
    \kappa+\tau]},\hspace{8cm}
$$
where $\kappa = (k_{1}, k_{2}, \dots, k_{m}) \in \Re^{m}$, $\tau = (t_{1}, t_{2}, \dots, t_{m}) \in
\Re^{m}$, Re$(a)> (m-1)\beta/2-k_{m}$ and Re$(b)> (m-1)\beta/2-t_{m}$. From Corollary \ref{cor2} and
\citet[Theorem 3.3.1]{dggs:15}, we have the following alternative definition:

\begin{definition}
The \emph{matrix multivariate} beta function is defined an denoted as:
\begin{eqnarray*}
% \nonumber to remove numbering (before each equation)
  \mathcal{B}_{m}^{* \ \beta}[a,k;b, \tau] &=& \int_{1-\tr\mathbf{B}>0}|\mathbf{B}|^{b-(m-1)\beta/2-1}
   (1-\tr\mathbf{B})^{a+k-1} q_{\tau}(\mathbf{B})(d\mathbf{B}) \\
   &=& \int_{\mathbf{R} \in\mathfrak{P}_{m}^{\beta}} |\mathbf{F}|^{b-(m-1)\beta/2-1}
   (1+\tr\mathbf{F})^{-(a+mb+k+\sum_{i=1}^{m}t_{i})} q_{\tau}(\mathbf{F})(d\mathbf{F})\\
   &=& \frac{\Gamma^{\beta}_{1}[a+k] \Gamma^{\beta}_{m}[b,\tau]}
    {\Gamma^{\beta}_{1}\left[a+mb+k+\sum_{i=1}^{m}t_{i}\right]} .
\end{eqnarray*}
\end{definition}
\end{remark}

Also, observe that, when $m = 1$, then $\tau = t$ and $\kappa = k$ and
$$
  \mathcal{B}_{1}^{\beta}[a,k;b,t]=  \frac{\Gamma^{\beta}_{1}[a+k] \Gamma^{\beta}_{1}[b+t]}
    {\Gamma^{\beta}_{1}\left[a+b+k+t\right]} = \mathcal{B}_{1}^{* \ \beta}[a,k;b,t]
$$

Finally observe that if in results in this section are defined $k =0$ and $\tau = (0, \dots,0)$, the
results in \citet{dggj:12} are obtained as particular cases.

\section{Singular value densities}\label{sec4}
In this section, the joint densities of the singular values of random matrix $\mathbf{R} \sim
\mathcal{P_{_{II}}R}_{n \times m}^{\beta,I}(\nu, k,\tau,1,\boldsymbol{0},\mathbf{I}_{n}, \mathbf{I}_{m})$
are derived. In addition, and as a direct consequence, the joint density of the eigenvalues of matrix
multivariate beta-Riesz type I distribution is obtained for real normed division algebras.

\begin{theorem}
Let $\delta_{1}, \dots, \delta_{m}$, $1 >\delta_{1}> \cdots > \delta_{m} > 0$, be the singular values of
the random matrix $\mathbf{R} \sim \mathcal{P_{_{II}}R}_{n \times m}^{\beta,I}(\nu, k, \tau, 1,
\mathbf{0}, \mathbf{I}_{n}, \mathbf{I}_{m})$. Then its joint density is
$$
   \frac{2^{m} \pi^{\beta m^{2}/2 + \varrho} }{\Gamma_{m}^{\beta}[\beta m/2]\mathcal{B}_{m}^{* \ \beta}
   [\nu\beta/2,k;n\beta/2, \tau]} \prod_{i=1}^{m} \left(\delta_{i}^{2}\right)^{(n-m+1)\beta/2-1/2}
  \left(1-\rho\sum_{i=1}^{m}\delta_{i}^{2}\right)^{\nu\beta/2+k-1}
  \hspace{1cm}
$$
\begin{equation}\label{svMT1}
\hspace{5cm} \times \ \prod_{i<j}^{m}\left(\delta_{i}^{2} - \delta_{j}^{2}\right)^{\beta}
  \frac{C_{\tau}^{\beta}(\mathbf{D}^{2})}{C_{\tau}^{\beta}(\mathbf{I}_{m})} \left(\bigwedge_{i=1}^{m}d\delta_{i}\right)
\end{equation}
for $1-\rho\sum_{i=1}^{m}\delta_{i}^{2}>0$. Where $\varrho$ is defined in Lemma \ref{lemsvd}, $\mathbf{D}
= \diag(\delta_{1},\dots, \delta_{m})$, and $C_{\kappa}^{\beta}(\cdot)$ denotes the zonal spherical
functions or spherical polynomials, see \citet{gr:87} and \citet[Chapter XI, Section 3]{fk:94}.
\end{theorem}
\begin{proof}
This follows immediately from (\ref{p2r}). First using (\ref{svd}), then applying (\ref{vol}) and
observing that, from \citep[Equation 4.8(2) and Definition 5.3]{gr:87} and \citet[Chapter XI, Section
3]{fk:94}, we have that for $\mathbf{L} \in \mathfrak{P}_{m}^{\beta}$,
$$
    C_{\tau}^{\beta}(\mathbf{Z}) = C_{\tau}^{\beta}(\mathbf{I}_{m})\int_{\mathbf{H} \in \mathfrak{U}^{\beta}(m)}
     q_{\kappa}(\mathbf{HZH}^{*})(d\mathbf{H}),
$$
\end{proof}

Finally, observe that $\delta_{i} = \sqrt{\eig_{i}(\mathbf{R}^{*}\mathbf{R})}$, where
$\eig_{i}(\mathbf{A})$, $i = 1, \dots, m$, denotes the $i$-th eigenvalue of $\mathbf{A}$. Let
$\lambda_{i} = \eig_{i}(\mathbf{R}^{*}\mathbf{R}) = \eig_{i}(\mathbf{B})$, observing that, for example,
$\delta_{i} = \sqrt{\lambda_{i}}$. Then
$$
  \bigwedge_{i=1}^{m} d\delta_{i} =  2^{-m} \prod_{i=1}^{m}
  \lambda_{i}^{-1/2} \bigwedge_{i=1}^{m} d\lambda_{i},
$$
the corresponding joint densities of $\lambda_{1}, \dots, \lambda_{m}$, $1 > \lambda_{1} > \cdots >
\lambda_{m}> 0$ is obtained from (\ref{svMT1}) as
$$
  \frac{\pi^{\beta m^{2}/2 + \varrho} }{\Gamma_{m}^{\beta}[\beta m/2]\mathcal{B}_{m}^{* \ \beta}
  [\nu\beta/2,k;n\beta/2, \tau]} \prod_{i=1}^{m} \lambda_{i}^{(n-m+1)\beta/2-1}
  \left(1 - \sum_{i=1}^{m}\lambda_{i}\right)^{\nu\beta/2+k-1}
  \hspace{1cm}
$$
$$
\hspace{4cm} \times \ \prod_{i<j}^{m}\left(\lambda_{i} - \lambda_{j}\right)^{\beta}
  \frac{C_{\tau}^{\beta}(\mathbf{G})}{C_{\tau}^{\beta}(\mathbf{I}_{m})} \left(\bigwedge_{i=1}^{m}d\lambda_{i}\right)
$$
for $1 - \sum_{i=1}^{m}\lambda_{i} > 0$, where $\mathbf{G} = \diag(\lambda_{1}, \dots,\lambda_{m})$.

\section{Conclusions}

As visual examples, different Pearson type II-Riesz densities for $m = 1$ are showed in figures
\ref{fig1} and \ref{fig2},

\begin{figure}[p]
  \begin{center}
   \includegraphics[width=10cm,height=9cm]{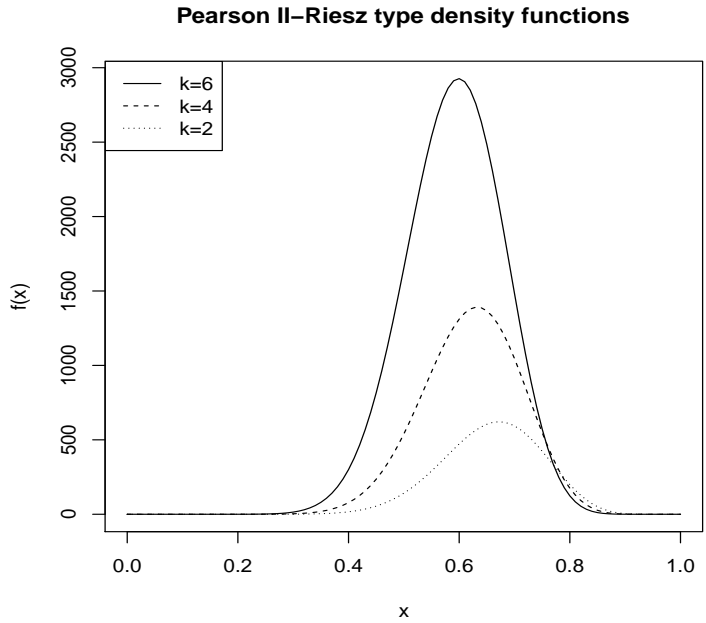}%
   \caption{With $\nu = 15$, $n = 18$ and $t = 7$}\label{fig1}
  \end{center}
\end{figure}

\begin{figure}[p]
  \begin{center}
   \includegraphics[width=10cm,height=9cm]{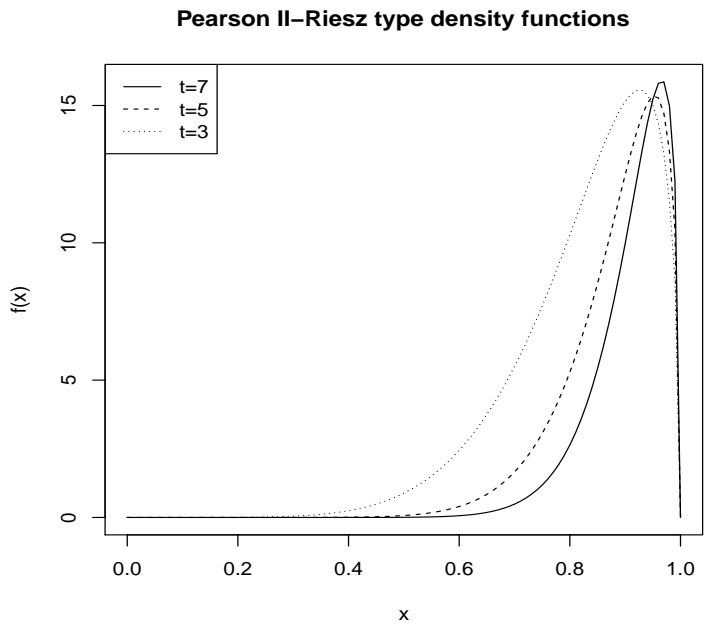}%
   \caption{With $\nu = 3$, $n = 18$ and $k = 0$}\label{fig2}
  \end{center}
\end{figure}

Recall that in octonionic case, from the practical point of view, we most keep in mind the fact from
\citet{b:02}, \emph{there is still no proof that the octonions are useful for understanding the real
world}. We can only hope that eventually this question will be settled on one way or another. In
addition, as is established in \citet{fk:94} and \citet{S:97} the result obtained in this article are
valid for the \emph{algebra of Albert}, that is when hermitian matrices ($\mathbf{S}$) or hermitian
product of matrices ($\mathbf{X}^{*}\mathbf{X}$) are $3 \times 3$ octonionic matrices.

\section*{Acknowledgements}
%The authors wish to thank the Editor and the anonymous reviewers for their constructive
%comments on the preliminary version of this paper.
This paper was written during J. A. D\'{\i}az-Garc\'{\i}a's stay as a visiting professor at the
Department of Statistics and O. R. of the University of Granada, Spain; it stay was partially supported
by IDI-Spain, Grants No. MTM2011-28962 and under the existing research agreement between the first author
and the Universidad Aut\'onoma Agraria Antonio Narro, Saltillo, M\'exico.

\bibliographystyle{plain}

\end{document}